\documentclass[11pt]{amsart}

\usepackage{amssymb,amstext,amsmath,amscd,amsthm,amsfonts,enumerate,graphicx,latexsym,accents}
\usepackage[usenames]{color}
\usepackage{diagrams}
\usepackage[colorlinks=true,linkcolor=blue,urlcolor=blue]{hyperref}

\usepackage[T5, T1]{fontenc}
\usepackage{enumitem}
\usepackage{hyperref}
\usepackage{aliascnt}
\usepackage{tikz}
\usetikzlibrary{cd}
\usetikzlibrary{matrix,arrows,decorations.pathmorphing}

\usepackage[all]{xy}

\usepackage{enumitem}

\newlist{myenumerate}{enumerate}{1}
\setlist[myenumerate,1]{label=(\roman*)}

\newtheorem{thm}{Theorem}[section]
\newaliascnt{lem}{thm}
\newtheorem{lem}[lem]{Lemma}
\aliascntresetthe{lem}

\newaliascnt{cor}{thm}
\newtheorem{cor}[cor]{Corollary}
\aliascntresetthe{cor}

\newaliascnt{prop}{thm}
\newtheorem{prop}[prop]{Proposition}
\aliascntresetthe{prop}

\newaliascnt{dfn}{thm}

\aliascntresetthe{dfn}

\newaliascnt{rem}{thm}

\aliascntresetthe{rem}

\theoremstyle{definition}

\newtheorem{conj}[thm]{Conjecture}
\newtheorem{ex}[thm]{Example}

\newtheorem*{claim*}{Claim}
\theoremstyle{remark}

\numberwithin{equation}{thm}
\def\Hom{\operatorname{Hom}}

\def\ord{\operatorname{ord}}

\def\lto{\longrightarrow}

\def\Ext{\operatorname{Ext}}

\def\syz{{\rm syz}}

\def\mod{\operatorname{mod}}
\def\im{\operatorname{Im}}
\def\ker{\operatorname{Ker}}

\def\m{\mathfrak{m}}

\def\n{\mathfrak{n}}

\def\depth{\operatorname{depth}}

\def\soc{\operatorname{Soc}}

\begin{document}

\title{Syzygies of the residue field over Golod rings}
\date{March 2, 2025}

\author[\fontencoding{T5}\selectfont D.T. C\uhorn{}\`\ohorn ng]{\fontencoding{T5}\selectfont \DJ o\`an Trung C\uhorn{}\`\ohorn ng} 
\address[\fontencoding{T5}\selectfont \DJ o\`an Trung C\uhorn{}\`\ohorn ng] {Institute of Mathematics, Vietnam Academy of Science and Technology, 18 Hoang Quoc Viet, 10072 Hanoi, Viet Nam.} \email{dtcuong@math.ac.vn}

\author[H.Dao]{Hailong Dao}
\address[Hailong Dao]{Department of Mathematics, University of Kansas, Lawrence, KS 66045}
\email{hdao@ku.edu}
\urladdr{https://www.math.ku.edu/~hdao/}

\author[D. Eisenbud]{David Eisenbud}
\address[David Eisenbud]{Department of Mathematics, University of California, Berkley, CA 94720}
\email{de@berkeley.edu}
\urladdr{eisenbud.github.io}

\author[T. Kobayashi]{Toshinori Kobayashi }
\address[Toshinori Kobayashi]{Department of Mathematics, School of Science and Technology, Meiji University, 1-1-1 Higashi-mita, Tama-ku, Kawasaki 214-8571, Japan.}
\email{tkobayashi@meiji.ac.jp}

\author[C. Polini]{Claudia Polini}
\address[Claudia Polini]{Department of Mathematics, Notre Dame University, South Bend, IN 46556}
\email{cpolini@nd.edu}
\author[B. Ulrich]{Bernd Ulrich}
\address[Bernd Ulrich]{Department of Mathematics, Purdue University, West Lafayette, IN 47907}
\email{bulrich@purdue.edu}

\thanks{HD was partly supported by Simons Foundation grant MP-TSM-00002378. DE is grateful to the
National Science Foundation for partial support through grant 2001649.  CP and BU were partially supported by NSF grants DMS-2201110 and DMS-2201149, respectively. DTC was funded by Vingroup Joint Stock Company and supported by Vingroup Innovation Foundation (VinIF) under the project code VINIF.2021.DA00030.
This material is partly based upon work supported by the National Science Foundation under Grant No. DMS-1928930, while four of the authors were in residence at SLMath in Berkeley, California, during the Special Semester in Commutative Algebra, Spring 2024.
}
\begin{abstract}
Let $(R,\m,k)$ be a Golod ring. We show a recurrence formula for high syzygies of $k$ in terms of previous ones. In the case of embedding dimension at most $2,$ we provided complete descriptions of all indecomposable summands of all syzygies of $k.$ 
\end{abstract}
\maketitle
\section{Introduction}

Since the seminal work of Hilbert significant advances have been made in  understanding the structure of finite free resolutions. However, much less is known about infinite free resolutions, which are quite common, as most minimal free resolutions over Noetherian local ring are infinite. Unfortunately, the standard techniques used to study finite free resolutions rarely apply to infinite resolutions.

This paper deals with minimal free resolutions of finitely generated modules over Noetherian local rings, with emphasis on the residue field. While there are numerous results and conjectures about Betti numbers  \cite{Gu67, Gu71, GU80, A90, Le90, AGP97, A96, MP15}, our focus instead is on the structure of syzygy modules and finiteness properties of these in general infinite resolutions.
 

We will focus on Golod rings. They  appear naturally in many contexts. Suppose that $R = S/I$ with $(S,\n)$ regular local (or graded)
of dimension $e$ and $I \subset \n^2$ (so that $e$ is the embedding dimension of $R$). Then $R$ is Golod, for example, if $I$ 
has codimension one \cite{Avramov2010}; or 
 if $e=2$ and $R$ is not a zero-dimensional complete intersection \cite{Sc64}; or if $R$ is a local ring
 of ``minimal multiplicity''  $e-\dim R +1$ \cite{A90}
 ; or if $R$ is graded and $I$ has linear resolution  \cite{BF85}  or is componentwise linear  \cite{HRW99}; or  if $I$ is a Borel fixed monomial ideal \cite{AH96, Pee96}; or if $I = J^{s}$ is a power (or a symbolic power) with   $s\ge 2$ \cite{HH13}. The Golod property is stable under factoring out a regular sequence that is part of a regular system of parameters of $S$ \cite{A90}.  

Let $(R,\m,k)$ be a Noetheran local ring of embedding dimension $e.$ We prove that if $R$ is Golod then every syzygy module of the $R$-module $k$ is a direct sum of copies of the first $e+1$ syzygy modules, $\syz_i^R(k)$ for $0 \leq i \leq e,$ and we give a recursive formula for the number of copies: 
 
\begin{thm}\label{generalGolod}
Let $(R,\m,k)$ be a Noetherian local ring of embedding dimension $e.$ Let $K_{\bullet}$ be the Koszul complex
of a minimal set of generators of $\m.$ If $R$ is Golod then 
$$
\syz_{e+1}^R(k) = \bigoplus_{j= 0}^{e-1} \syz_j^R(k)^{h_{e-j}}
$$
and, more generally,
for every $i>e,$
$$\syz^R_i(k)=\bigoplus_{i-e-1 \leq j \leq i-2} \syz^R_j(k) ^{ h_{i-j-1}}\, , $$
where $h_{i-j-1} = \dim_k(H_{i-j-1}(K_{\bullet})).$ 
\end{thm} 
This structural result provides a new explanation of Golod's well-known formula~\cite{Golod} for the ranks of the free modules in the minimal resolution of $k,$ which is an immediate consequence.
\autoref{generalGolod} implies that the direct sum decompositions into indecomposables for the first $e+1$ syzygy modules determine such decompositions for all syzygy modules of $k.$
This is in stark contrast to the case of a zero-dimensional Gorenstein ring $R$ with $e \geq 2,$ where the (infinitely many) syzygy modules of $k$ are all indecomposable and non-isomorphic.

We will focus on the case $e=2,$ where the Golod assumption in \autoref{generalGolod}  simply means that 
$R$ is not a zero-dimensional complete intersection \cite{Sc64}. In this case we will give an explicit description of the direct sum decompositions into indecomposables
of the syzygy modules $\syz_i^R(k)$ for all $i.$ By \autoref{generalGolod} it suffices to do this for $\syz_{1}^{R}(k) = \m$ and $\syz_{2}^{R}(k) = \syz_{1}^{R}(\m).$

All our results are preserved and reflected by completion. For example, the number of summands of a finitely generated $R$-module $M$ that are isomorphic to the
residue field $k$ is $\dim_{k}(\soc(M)/(\m M \cap \soc(M))$, and this does not change upon completion. Thus we may assume, without loss of generality, that $R = S/I$,
where $S$ is a regular local ring and  $I\subset \n^2$. Let $e:={\rm dim} \, S$ be the embedding dimension of $R.$ This will be our notation throughout this paper.

Assume that $e=2$ and $I\subset \n^2$ is not
an $\n$-primary complete intersection. We prove that $\syz_1^R(k) = \m$ is decomposable if and only
if $xy \in I$ for some regular system of parameters $x,y$ of $S$ (see \autoref{mindec}). In this case 
$\m = R/(0:x)\oplus R/(0:y)$ is the direct sum decomposition into indecomposables. For $\syz_2^R(k)=\syz_1^R(\m)$
we obtain (see \autoref{mdec}):

\begin{prop}  If $\m$ is decomposable then
$\syz_2^R(k) =   \m \oplus k^a,$ where
$$
a=\begin{cases}
			2 & \text{if $\dim R =0$}\\
            1 & \text{if $\depth R = 0$ and $\dim R =1$}\\
            0 & \text{if $R$ is Cohen-Macaulay of dimension 1} \, .
            \end{cases}
$$
\end{prop}

It remains to treat the more general, and more difficult, case where $\m$
is indecomposable. We only need to give the direct sum decomposition into indecomposables of $\syz_{2}^{R}(k) = \syz_1^{R}(\m).$ 
 The following result combines \autoref{syz and dual0} and \autoref{Nindec}:

\begin{thm}\label{THM1.3bis} If $\m$ is indecomposable, then 
$$\syz_2^R(k)=\syz_1^R(\m) \cong \Hom(\m, R) \cong N \oplus k^{a} \, ,$$
where $a= \dim_{k}\left(\frac{\n(I:\n)}{\n I}\right)$ and $N$ is indecomposable.
\end{thm}

We can make the decomposition of $\syz^{R}_{1}(\m)$ in  \autoref{THM1.3bis} very explicit. Let $x,y$ be minimal generators 
of $\n$ and $h_1, \ldots, h_n$ be minimal
generators of $I,$ write $h_i=f_ix+g_iy,$ and let $a$ be as in \autoref{THM1.3bis}.
 Choose generators of $I$ so that the images of the last $a$ generators $h_i$
form a $k$-basis of $\frac{\n(I:\n)}{\n I}$ and choose the corresponding $f_i$ and $g_i$ in $I:\n.$  
With this notation we will show that $\syz_1^R(\m)$ is the submodule of $R^2$ generated by the 
columns of the matrix
$$\begin{pmatrix}
\overline{y}&\overline{f_{1}}&\dots&\overline{f_{n}} \\
-\overline{x} &\overline{g_{1}}&\dots&\overline{g_{n}}
\end{pmatrix}  ,
$$
where $^{-}$ denotes images in $R.$
Now let $N$ and $N'$ be the submodules  of $R^2$ 
generated by the first 
$t+1-a$ columns and by the last $a$ columns, respectively. For these particular submodules $N$ and $N'$ we have:

\begin{cor}\label{THM1.3} If $\m$ is indecomposable, then 
$$ \syz_2^R(k)=N \oplus N' \, ,$$
where $N' \cong k^{a}$ and $N$ is indecomposable.

\end{cor}

We generalize the first isomorphism in \autoref{THM1.3bis} to second syzygies of some cyclic modules other than $k.$ For instance, if 
$R$ is Artinian and $J$ is an ideal so that the ring $R/J$ is a complete
intersection, then using linkage we show that, if ${\rm Fitt}_2(I) R \subset J$, then
$\syz_2^R(R/J) \cong J^*:= \Hom(J, R)$ 
 (see \autoref{syz and dual}).

A consequence of these results is that at most three non-isomorphic indecomposable modules appear in the
direct sum decompositions of all the syzygy modules of $k,$ and that these indecomposable 
modules are summands of $k,$ $\m,$ $\m^*$ (see  \autoref{main}). In the case $e=2$ we are also able to characterize when, for any given $i,$ the syzygy module $\syz_i^R(k)$
is indecomposable (see \autoref{decomp of syzygies}).

In  experiments with rings of embedding codimension $>2$ we have seen an analogous phenomenon:

\begin{conj}
 If $(R,\m,k)$ is a local Golod ring of embedding codimension $e$, then there is a set of at most $e+1$ indecomposable
 modules from which every $R$-syzygy of $k$ may be built as a direct sum.
\end{conj}
 
\vspace{.2cm}

Our results were suggested by Macaulay 2 computations \cite{M2}, performed at an AIM meeting in September 2023 with the help of
Mahrud Sayrafi and Devlin Mallory, using their \emph{DirectSummands} package. Without this support we might never have guessed
that the results of this paper could be true.

\section{$\syz^{R}_{1} (\m)$ is $\m^{*}$}\label{sec2}

\begin{thm} \label{syz and dual0}
 If $(R,\m)$ is a Noetherian local ring of embedding dimension 2 that is not a zero-dimensional complete intersection, then $\syz_{1}^{R}(\m) \cong \m^{*}.$
\end{thm}

We postpone the proof until after \autoref{syz and dual}.
 \def\Fitt{{\rm Fitt}}

 \begin{thm}\label{3 equi}
Let $S$ be a ring and let  $I\subset J$ be ideals of $S$. Set $R = S/I$ and write $(-)^{*} = \Hom_{S}(-, R)$ for the $R$-dual.
The following conditions are equivalent$\, :$
\begin{enumerate}[label={$($\arabic*\,$)$}]
 \item The dual 
 $
(JR)^{*} \to J^{*}
$
of the natural surjection $J \to JR$ is an isomorphism.

 \item The restriction map $J^{*}\to I^{*}$ is 0.
 \item The natural map $\Ext^{1}_{R}(S/J,R) \to \Ext^{1}_{S}(S/J,R) $ is an isomorphism.
\end{enumerate}
If these conditions are satisfied for $J$ then they are satisfied for any ideal containing $J.$
\end{thm}

\begin{ex}\label{easy 3 equi} The conditions (1)--(3) of Theorem~\ref{3 equi} are satisfied if the $R$-ideal
$J/I = JR$ contains a nonzerodivisor. 
The natural exact sequence of $R$-modules
$$ 0\to I/IJ \to J \otimes_S R \to JR \to 0
$$ gives an exact sequence
$$
0\to (JR)^{*} \to J^{*} \to \Hom(I/IJ, R).
$$
But $\Hom(I/IJ,R) = 0$ since the $R$-module $I/IJ$ is annihilated by $JR$.
Thus the map $(JR)^{*} \to J^{*}$ is
an isomorphism as in condition (1) of Theorem~\ref{3 equi}.
\end{ex}
 
\begin{proof}[Proof of \autoref{3 equi}]
\noindent (1)$ \Longleftrightarrow$ (2): Dualizing the exact sequence
$$
0\to I\to J\to (J/I = JR) \to 0
$$
yields the result.

 \noindent (1) $\Longleftrightarrow$ (3): We have a diagram 
 \begin{tiny}
 $$
\begin{diagram}
0&\rTo &JR &\rTo &R& \rTo &R/JR&\rTo& 0 \\
&&\uTo&&\uTo&&\uTo^{\cong}\\
0&\rTo &J &\rTo &S& \rTo &S/J&\rTo& 0 
\end{diagram}.
$$
\end{tiny}
\vskip .2cm

\noindent
Dualizing into $R,$ we get the diagram
\begin{tiny}
$$
\begin{diagram}
0&\lTo &\Ext_{R}^{1}(R/JR, R) &\lTo &(JR)^{*} &\lTo &R& \lTo &(R/JR)^{*}&\lTo& 0 \\
&&\dTo&&\dTo&&\dTo^{=}&&\dTo^{\cong}\\
0&\lTo &\Ext_{S}^{1}(S/J, R) &\lTo &J^{*} &\lTo &R& \lTo &(S/J)^{*}&\lTo& 0 
\end{diagram}.
$$
\end{tiny}

\vskip .2cm

\noindent
The equivalence now follows from the ``five lemma''.

The last statement follows at once from condition (2).
 \end{proof}

\begin{prop}\label{relation of dual to syz} With notation as in \autoref{3 equi}, if $J$ is generated by an $S$-regular sequence $x,y,$ 
then $ J^{*}\cong \syz_{1}^{R} (JR)$ via the map $f \mapsto (f(y), -f(x)).$ In particular, if $J^{*}= (JR)^{*}$ then $f$ indices an isomorphism
$$
\syz_{1}^{R} (JR) \cong (JR)^{*}.
$$
\end{prop}
\begin{proof}
We may write $\syz_{1}^{R}(JR) = \{(\overline a,\overline b) \mid ax+by \in I\}$ where $\overline{\phantom a }$ denotes images in $R.$
Consider the maps
$$
 S \rTo^{
\begin{pmatrix}
 y\\
 -x
\end{pmatrix}
}
S^{2} \rTo^{
\begin{pmatrix}
 -b&a
\end{pmatrix}
}
 R.
$$
The composition is 0 if and only if $ax+by \in I,$ and since $J$ is the cokernel of $\begin{pmatrix}
 y\\
 -x
\end{pmatrix}
,$ this is the condition that $(-b, a)$ induces a homomorphism $J \to R.$
\end{proof}

\begin{thm}\label{syz and dual}
 Let $S$ be a Noetherian local ring and let $x,y$ be an $S$-regular sequence. 
 Let $I\subset S$ be an ideal of projective dimension one, so that we may write $I = aI'$, where $I'$ is perfect of grade 2 and $a$ is a nonzerodivisor. 
 
 If
  $I'\subset J := (x,y)$ then
conditions $ (1)-(3)$ of \autoref{3 equi}
are equivalent to the condition that $a\cdot \Fitt_{2}(I) \subset J.$ 
\end{thm}


\begin{proof} Write $R=S/I$ and $(-)^{*} =\Hom_{S}(-, R).$  We denote images in $R$ by $^{-}.$

We will show that the restriction map $\rho: J^{*} \to I^{*}$ is 0 
if and only if $a\cdot \Fitt_{2}(I) \subset J.$

Let $h'_1, \ldots, h'_n$ be generators of $I'$ so that the elements $h_{i} = ah'_{i}$ generate $I$. 
Extending the ground field if necessary, we may assume that $h_i', h_j'$ form a regular sequence for every $i \neq j$ 
and that $h_i',a$ form a regular sequence for every $i.$ Let 
$$\varphi'=
\begin{pmatrix}
f'_{1}&\dots&f'_{n} \\
g'_{1}&\dots&g'_{n}
\end{pmatrix}
$$
be a matrix with entries in $S$ satisfying $\begin{pmatrix} h'_1 & \dots & h'_n\end{pmatrix}=\begin{pmatrix} x & y\end{pmatrix} \cdot \varphi'$, and
let $\varphi = a\varphi'$.  

We first prove that $\rho=0$ if and only if $aI_2(\varphi')\subset I'.$ Consider presentations of $J$ and $I$ with respect to the generators $x,y$ and $h_1, \ldots, h_n,$ respectively, and a morphism between them, 
\begin{tiny}
$$
\begin{diagram}
S&\rTo^{\begin{pmatrix}y\\-x\end{pmatrix}} & S^2& \rTo & J & \rTo & 0\\
\uTo&&\uTo^{\varphi}&&\uTo&&\\
S^{n-1}&\rTo &S^n&\rTo& I &\rTo& 0\, .\\
\end{diagram}
$$
\end{tiny}
\vskip .2cm
\noindent
Dualizing into $R$ we obtain a commutative diagram with exact rows 
\begin{tiny}
$$
\begin{diagram}
&&R^{n+1}\\
&&&\rdTo^{\overline \psi}\\
0&\rTo &J^{*}&\rTo& {S^2}^{*} ={R^2}^*&\rTo^{(\overline{y}, -\overline{x})}& S^{*}=R^*\\
&&\dTo^{\rho}&&\dTo^{\varphi^{*}=\overline{\varphi}^{*}}&&\dTo\\
0&\rTo &I^{*}&\rTo& {S^n}^{*} ={R^n}^*&\rTo& {S^{n-1}}^{*}={R^{n-1}}^*\, .\\
\end{diagram}
$$
\end{tiny}
\vskip .2cm

\noindent
Thus $\rho=0$ if and only if $\overline{\varphi}^*$ is zero when restricted to the image of $J^*.$ This image is the syzygy module of $\overline{y}, -\overline{x}$ in ${R^2}^*,$ which in turn is generated by the columns of the matrix $\overline{\psi},$ where
$$
\psi=\begin{pmatrix}
x &g_{1}&\dots&g_{n} \\
y & -f_{1}&\dots&-f_{n}
\end{pmatrix}\, .
$$
Therefore $\rho=0$ if and only if $\overline{\varphi}^{*}\overline{\psi}=0.$ 

%
Since 
\begin{small}
 $$
\varphi^t \psi = 
\begin{pmatrix}
 h_{1}& 0 & \Delta_{1,2} &\dots& \Delta_{1,n}\\
 h_{2}& -\Delta_{1,2}& 0 & \dots& \Delta_{2,n}\\
 \vdots &\vdots &&\ddots &\vdots\\
 h_n& -\Delta_{1,n} & -\Delta_{2,n} & \dots & 0
\end{pmatrix}
$$
\end{small}
where $\Delta_{i,j}$ is the determinant of the submatrix of $\varphi$ involving columns $i,j,$ we see that $\rho=0$ if and only if $I_2(\varphi)\subset I,$
 and this is the case if and only if $aI_{2}(\varphi') \subset I'$, as claimed.

Let $\Delta_{i,j}'$ be the minor of $\varphi'$ involving columns $i,j. $ 
Since $\begin{pmatrix} h_i' &  ah_j'\end{pmatrix}=\begin{pmatrix} x & y\end{pmatrix} \cdot \begin{pmatrix} f_i'& af_j' \\g_i' & ag_j'\end{pmatrix}$ and $h_i', ah_j'$ form a regular sequence, a theorem of Gaeta (see for instance \cite[Example 3.2(b)]{AN}) gives
$$
(h_{i}', ah_{j}'):J = ( h_{i}', ah_{j}', a\Delta_{i,j}')\, .
$$
As $I'$ is perfect of grade 2, it follows that $a \cdot \Delta_{i,j}'\in I'$ if and only if $(h_{i}', a h_{j}'):I' \subset J$ by the symmetry of linkage.

By the same theorem of Gaeta, the link  $(h_{i}', a h_{j}'):I' $ of $I'$ is generated by $h_i'$ and $a$ times the $n-2$ minors of the presentation matrix of $I' \cong I$ with
rows $i$ and $j$ deleted.  This proves  that $a \cdot I_2(\varphi')\subset I'$ if and only if  $a \cdot \Fitt_{2}(I) \subset J.$
\end{proof}

\begin{proof}[Proof of \autoref{syz and dual0}] We apply the previous results with $J=\n.$
 If $I$ is principal, we use Example~\ref{easy 3 equi} and \autoref{relation of dual to syz}. If $I$ is not principal, we may write $I = aI'$ with $I'$ perfect of grade 2. We see from \autoref{relation of dual to syz} and \autoref{syz and dual} that the result holds unless both ($a$) and $\Fitt_2(I)$ are unit ideals. If $a$ is a unit, then $I= I'$ has grade 2. If in addition $\Fitt_2(I)=S,$ then $I$ is a zero dimensional complete intersection. \end{proof}

\vspace{.001cm}

\section{The decomposition of $\m$}\label{sec3}

In this section $(S,\n,k)$ denotes a regular local ring of dimension 2 and  $I\subset S$ is an ideal. Write $R= S/I$ and $\m = \n R.$

\begin{lem}\label{class-ideals} Suppose that $\n = (x,y).$ If $xy\in I$ then $I$
can either
 be written as
\begin{enumerate}[label=$($\alph*\,$)$]
 \item $I = (xy, ux^{\alpha}+vy^{\beta})$ where $u,v$ are each either 0 or units and $\alpha,\beta$ are non-negative integers;  or as
 \item $(xy, x^{\alpha}, y^{\beta}),$ where $\alpha$ and $ \beta$ are positive.
\end{enumerate}

\end{lem}

\begin{proof}
If $I$ has codimension 1 then $I$ has a proper common divisor, which we may take to be $x.$ Writing $I = x(J+(y))$, we see that either $I = (xy)$ or $I = (xy, x^{\alpha})$
for some $\alpha\geq 1$ because $S/(y)$ is a discrete valuation ring with parameter $x.$

Any element of an Artinian local ring can be written as a polynomial in the generators of the maximal ideal with unit coefficients.
In particular, if $I$ has codimension 2, then any element of $S/(xy, \n I)$ is the image of an element of the form $f = ux^{\alpha}+vy^{\beta},$ where each of $u$ and $v$ is either 0 or a unit of
$S$ and $\alpha, \beta$ are non-negative. Note that if $u\neq 0$ then, modulo $xy,$ every $x^{\mu}$ with  $\mu > \alpha$ is a multiple of $f$ and similarly for $v$ and $y.$

If $I/(xy)$ is principal, then we may write $I = (xy, ux^{\alpha}+vy^{\beta}),$ and we are done. Otherwise, modulo $xy,$
 we may write two of the generators of $I$ as $ux^{\alpha}+vy^{\beta},\ px^{\gamma}+qy^{\delta},$ where $u$ and $q$ are units and $\alpha$ and $\delta$ are minimal.
 Thus we may assume that $p=0,$ and $v=0,$ so $I= (xy, x^{\alpha}, y^{\delta}).$ Notice that $\alpha, \beta$ have to be positive.
\end{proof}
 
\begin{thm}\label{mindec} 
The following are equivalent$\, :$
\begin{enumerate}[label={$($\arabic*\,$)$}]
 \item The module $\m$ is decomposable.
 \item We may write $\n = (x,y)$ with  $xy\in I$ and $R$ is neither  a discrete valuation ring nor a zero-dimensional complete intersection.
 \item  We may write $\n = (x,y)$ in such a way that $I = (xy, ux^{\alpha}, vy^{\beta}),$ where each of $u,v$ is a unit of $S$ or 0 and $\alpha,\beta$ are $ \geq 2.$
\end{enumerate}
In this case $\m\cong R/(0:x) \oplus R/(0:y).$
\end{thm}

\begin{proof}
\noindent $(1) \Rightarrow (2)$ 
If $\m$ is decomposable, then it has to decompose as $\m=xR\oplus yR$ where $\m = (x,y).$ This implies $xy\in I.$ 
Every ideal of a domain is indecomposable, and every non-zero ideal of a zero-dimensional local Gorenstein ring contains the socle, and thus is indecomposable. 

\noindent $(2) \Rightarrow (3)$ This follows from \autoref{class-ideals} because $R$ is not a discrete valuation ring or a zero-dimensional complete intersection. 

\noindent $(3) \Rightarrow (1)$ One easily check that $(I, x) \cap (I,y) =I.$
\end{proof}
 
\vspace{.1cm}

\section{The decomposition of $\syz^R_2(k) = \syz^{R}_{1} (\m)$}\label{sec4}

Again in this section $(S,\n,k)$ is a regular local ring of dimension 2 and $I$ is an ideal. Set  $R= S/I$ and $\m = \n R.$ We denote images in $R$ by $^{-}.$

\begin{prop}\label{mdec} If $\m$ is decomposable then
$\syz^R_2(k) = \syz^{R}_{1} (\m)\cong  \m \oplus  k^{a},$ where
$$
a=\begin{cases}
			2 & \text{if $\dim R =0$}\\
            1 & \text{if $\depth R = 0$ and $\dim R =1$}\\
            0 & \text{if $R$ is Cohen-Macaulay of dimension 1} \, .
            \end{cases}
$$
\end{prop}

\begin{proof} We apply the analysis of~\autoref{mindec}.
This shows that $I\subset \n^{2}$ and we may assume $\m=\overline{x}R\oplus \overline{y}R.$ Furthermore, if $\dim R =0$ then we can write  $I=(x^{\alpha},xy,y^{\beta}),$ where $\alpha, \beta$ are  $\ge 2$. In this case $\overline{x}R\cong R/(\overline{x}^{\alpha-1},\overline{y})$ and $\overline{y}R\cong R/(\overline{x},\overline{y}^{\beta-1}).$ Thus $ \syz_1^{R}(\m) = (\overline{x}^{\alpha-1},\overline{y}) \oplus (\overline{x},\overline{y}^{\beta-1})\cong k \oplus \overline{y}R \oplus \overline{x}R \oplus k \cong \m \oplus k^2$ since  $\overline{x}^{\alpha-1}, \overline{y}^{\beta-1}$ are in the socle of $R.$ 

If $\depth R = 0$ and $\dim R =1,$ then \autoref{mindec} shows that
we may write $R = S/(xy, x^{\alpha})$ with $\alpha \ge 2.$ Now $\overline{x}R \cong R/(\overline{y}, \overline{x}^{\alpha-1})$ and $\overline{y}R \cong R/\overline{x}R,$ so  $\syz_{1}^{R}(\m) =(\overline{y},\overline{x}^{\alpha-1}) \oplus \overline{x}R \cong \overline{y}R \oplus k \oplus \overline{x}R \cong \m \oplus k$ because
$\overline{x}^{\alpha-1}$ is in the socle of $R.$

Finally, if $R$ is Cohen-Macaulay of dimension 1, then $I=(xy)$ by \autoref{mindec} and all the modules $\syz_{i}^{R}(k)$ for $i\geq 1$ are isomorphic.
\end{proof}

Given generators $x, y$ of $\n$ and $h_1, \ldots, h_n$ of $I$ we consider, as in the proof of \autoref{syz and dual}, a $2 \times (n+1)$ matrix with entries in $S$
$$L=\begin{pmatrix}
 L_0 & L_{1}&\ldots&L_{t}
\end{pmatrix}= \begin{pmatrix} y & f_1 & \ldots & f_n \\-x & g_1 & \ldots & g_n \end{pmatrix}
$$
such that $\begin{pmatrix} h_1 &\ldots & h_n \end{pmatrix}= \begin{pmatrix} x& y \end{pmatrix} \begin{pmatrix} f_1 & \ldots & f_n \\g_1 & \ldots & g_n \end{pmatrix}.$

\begin{lem}\label{one quadric}
 Suppose that $\m$ is indecomposable, and let $L$ be a matrix as above.
If $R$ is not a complete intersection, then$\,:$
\begin{enumerate}[label={$($\alph*\,$)$}]
 \item $I+\n^{3}$ does not contain any element $xy$ such that $\n = (x,y).$
\item $\dim_{k}(I+\n^{3})/\n^{3} \leq 1.$
\item There exists a choice of generators  $x,y$  of $\n$ and a choice of $f_i, g_i$ such that the entries of every column of the form
$L_{0}+\sum_{i>0}\lambda_{i} L_{i}$ generate $\n$ for all $\lambda_i \in S.$
\end{enumerate}
\end{lem}

\begin{proof} Since $R$ is not a complete intersection, we must have $I\subset \n^{2}$.

\noindent {\bf (a):}  Suppose first that $\dim R = 1.$ Since $S$ is factorial and $I$ is not a complete intersection,
 we may write $I = aI'$ where $I'$ is an ideal of codimension 2. If $I$ contains an element of order 2, then $a$ must have
 order 1. By condition (2) of \autoref{mindec}, any element of $I'$ that has order 1 must be a multiple of $a$,
completing the proof in the 1-dimensional case.

Now assume that $\dim R=0.$ Suppose that  $I$ contains an element $xy+f$ with $\ord f \geq 3$ such that $\n = (x,y).$ If $\n^{p}\subset I$ and 
$\ord f\geq p$ then $xy\in I,$ which is impossible by \autoref{mindec}(2). Otherwise, suppose there is
an expression $xy+f\in I$  such that $\n = (x,y)$ with order $\ord f$ maximal and $<p.$ 

We may write
$f = xf_{1}+yf_{2} +g$ with $\min\{\ord f_{1}, \ord f_{2}\} \geq \ord f-1$ and  $\ord g > \ord f.$  Thus $xy+f = (x+f_{2})(y+f_{1})+ (g- f_{1}f_{2}).$ Note that
$\ord (g-f_{1}f_{2})>\ord f.$ We may replace $x,y$ by $x+f_{2}, y+f_{1},$ thus increasing the order of $f,$ a contradiction.

\vskip .2cm
\noindent{\bf (b):} Suppose on the contrary that $ax^{2}+bxy+cy^{2}, a'x^{2}+b'xy+c'y^{2}$ are
linearly independent elements of $I+\n^{3}/\n^{3},$ where
$x,y$ are generators of $\n/\n^{3}$ and the coefficients $a,\dots, c'$ are in $k.$ By taking a linear combination, we may assume that $a=0,$ in which
case we are done by part (a) unless also $b=0.$ If on the other hand $b=0,$ then $c\neq 0,$ so we may assume that $c'=0.$ Now we are done unless $b'= 0.$ If $b'=0,$ then $x^{2}$ and $y^{2}$
are in $I+\n^{3}/\n^{3},$ but $xy$ is not by part (a). Thus the associated graded ring of $R,$ and with it $R$ itself, is a zero-dimensional complete intersection,
a contradiction. This shows that $\dim_{k}(I+\n^{3})/\n^{3} \leq 1,$ completing the proof.

\vskip .2 cm
\noindent{\bf (c):} By part (b) the quotient $(I+\n^{3})/\n^{3}$ is cyclic. If $I+\n^3=(\ell^2) +\n^3$ for some element $\ell$ of order 1, we choose generators $x=\ell, y$ for $\n.$ Otherwise we make an arbitrary choice. Furthermore, we may choose $f_i$ and $g_i$ to be in $\n^2$ for $i>1.$ If  $I\subset \n^3$ we also choose $f_{1}$ and $g_1$ to be  in $\n^2.$ If $I+\n^3=(x^2) +\n^3,$ we choose  $f_1\equiv x \mod \,\n^2$ and $g_1 \in \n^2.$
%
%
 

Now consider the $2\times 2$ matrix
$$
L' :=\begin{pmatrix}
 L_{0} &L_{1}'
\end{pmatrix}:=
\begin{pmatrix}
y&\sum_{i>0}\lambda_{i} f_{i} \\
-x &\sum_{i>0}\lambda_{i} g_{i}
\end{pmatrix}.
$$
If $\lambda_{1}\in \n$ or $I\subset \n^3,$ then $L_{0}+L_{1}' \equiv L_{0}  \mod \,  (\n^2\oplus \n^2)$ and the claim follows. Thus we can assume that $\lambda_{1}$ is a unit
and $I \not\subset \n^3.$ In this case 
${\rm det} (L') \not\in \n^3.$ Moreover ${\rm det}(L') \in I$ by the definition of the matrix $L.$ The determinant of $L'$ is also the determinant of the $2\times 2$ matrix
$(L_{0}+L_{1}' \ \ L_{1}').$ If the entries of $L_{0}+L_{1}'$ were linearly dependent modulo $\n^2,$ then the determinant would factor modulo $\n^3.$ Therefore $I+\n^3=(x^2)+\n^3$ by part (a). 
But then, 
modulo $\n^{2}$ the matrix $L'$ must be
$$
\begin{pmatrix}
y &x\\
-x&0
\end{pmatrix}.
$$
\vskip .1cm
\noindent
Thus the entries of $L_{0}+L_{1}'$ generate $\n$ as claimed.
 \end{proof}

The significance of the matrix $L$ considered in \autoref{one quadric} is that the columns of $\overline{L}$ are obviously a generating set of $\syz_1^R(\m),$  and even a minimal generating set by \cite[Satz 5]{Sc64}. In particular $\mu(\syz_{1}^{R}(\m))=\mu(I) +1,$ where $\mu(\cdot)$ denotes minimal number of generators. 

\begin{thm} \label{Nindec} Suppose that $I\subset \n^{2}$.
Write
$\syz_1^R(\m)=N \oplus N'$  where $N' \cong k^a$ and $N$ has no $k$-summands. If $\m$ is indecomposable then$\,:$
\begin{enumerate}[label={$($\alph*\,$)$}]
 \item  
 $a= \dim_{k}\left(\frac{\n(I:\n)}{\n I}\right)$ and $\mu(N)=\mu(I)+1-a\ge 1.$ 
 \item $N$ is indecomposable.
\end{enumerate}

\end{thm}

\begin{proof} 


\smallskip 
We may assume that $I\not=0.$ We fix generators $x,y$ of $\n$  and the corresponding embedding $Z:=\syz_1^R(\m)\subset R^2\, ,$ and we use the notation introduced before \autoref{one quadric}. 

\smallskip
\noindent {\bf (a):} Since $\mu(Z)=\mu(I) +1,$ we have  $\mu(N) = \mu(I)+1-a.$ 

 Notice that $$a=\dim_k \left(\frac{\soc Z}{\m Z\cap \soc Z}\right) .$$
Thus it suffices to prove that
$$\frac{\soc Z}{\m Z\cap \soc Z}\cong \frac{\n(I:\n)}{\n I}\, .
$$

\vspace{.02cm}

\noindent
To this end we define an $R$-linear map $\psi$ as the composition of the maps
$$
\soc Z \lto \frac{Z}{RL_0} =H_1(x,y; R) \stackrel{\sim}{\lto} \frac{I}{\n I}\,. $$
Notice that $\psi((f,g))=(xF+yG)+\n I,$ where $F,$ $G$ are preimages of $f,g$ in $S.$ As $\soc Z=\soc R^2,$ it follows that $\im \psi=\frac{\n(I:\n)}{\n I}.$ Clearly $\ker \psi =RL_0\cap \soc Z.$

Thus it remains to prove 
$$RL_0\cap \soc Z=\m Z \cap \soc Z\, .$$
The right hand side is in the left hand side, because  $ \frac{Z}{RL_0}=H_1(x,y; R)$ and therefore $\m Z \subset RL_0.$ As to the converse, the indecomposibility of $\m$ implies  that $I \not= \n^2,$ hence $\m L_0\not=0.$ Therefore $RL_0\cap \soc Z \subset \m L_0 \subset \m Z.$



\vskip .2cm

\noindent {\bf (b):} If $R$ is Gorenstein, then $Z$ is indecomposable. If $I=0$ this is obvious and otherwise it follows from the fact that  syzygies of indecomposable maximal Cohen-Macaulay 
modules over local Gorenstein rings are indecomposable. Thus the assertion of (b) holds and we may assume that $R$ is not Gorenstein.

Since $Z=N \oplus N'$ and $\m N'=0,$ we have $\m N=\m Z.$ As shown above $\m Z \subset RL_0.$ But $\overline{L_0}$ is a minimal generator of $Z,$ hence $\m Z\subset \m L_0,$ and therefore $\m Z=\m L_0.$ Finally, $RL_0 \cong R/(0:\m),$ so $\m L_0\cong \m/(0:\m).$ Putting these facts together, we have
\begin{equation}\label{importantEQ}
 \m\, N \cong \m/(0:\m) .
\end{equation}
Since $N$ does not have $k$ as a direct summand, the indecomposability of $N$ follows from the indecomposability of $\m N\cong \m/(0:\m),$ so
it suffices to treat the cases where the maximal ideal $\m/(0:\m)$ of $S/(I:\n)$ is decomposable.
By \autoref{mindec}(3) this is the case if and only if for a suitable choice of $x$ and $y$ one has $I:\n = (xy, ux^{\alpha},vy^{\beta}),$ where each of $u,v$ is a unit or 0 and both of $\alpha,\beta$ are $\geq 2.$

We first show that in this case the module $N$ can be generated by 2 elements. Set $I' = \n\, (I:\n).$ It suffices to prove that $\dim_{k}(I/I')\le 1$ because part (a) gives $\mu(N)=1+\mu(I)-\dim_{k}(I'/\n I)=1+\dim_{k}(I/I').$

Suppose first that $\dim R = 1.$  In this case we may assume that $I:\n = (xy, ux^{\alpha}),$ hence $ I \subset (xy, ux^{\alpha}).$ 
By \autoref{mindec}(2) the ideal $I$ contains no product of two elements that generate the maximal ideal of $S,$ so no element of the form 
$xy + \lambda x^{\alpha}$ with $\lambda\in S$ can be in $I.$ Thus 
$ I\subset \n (xy) + ( ux^{\alpha}) =  (x^{2}y, xy^{2}, ux^{\alpha}) =: I''.$ As $I'=(x^2y,xy^2, ux^{\alpha+1}),$ it follows that  $\dim_k(I''/I') \leq 1.$ Now the containments  $I' \subset I \subset I''$ show that $\dim_k(I/I')\le 1.$

Now assume that $\dim R=0,$ thus $I:\n=(xy, x^{\alpha}, y^{\beta}),$ where $\alpha, \beta$ are $\ge 2.$ If  $\alpha = \beta = 2$ then $I:\n=\n^2$ and hence $I'=\n^{3}.$ Therefore $\dim_{k}(I/I')\le 1$ by \autoref{one quadric}(b).
Finally, without loss of generality, we can assume that $\alpha\ge 3.$ We have
$I \subset I:\n = (xy, x^{\alpha}, y^{\beta}).$ Since $\alpha \geq 3$,  \autoref{one quadric}(a) shows that $I$ cannot contain an element
of the form $xy + \lambda x^{\alpha} + \mu y^{\beta}.$ Thus $I \subset \n(xy) + (x^{\alpha}, y^{\beta})=(x^2y,xy^2, x^{\alpha}, y^{\beta})=:I''.$  As $x^{\alpha-1}\notin I:\n$ but $x^{\alpha-1}\in I'':\n$ because $\alpha\ge 3,$ the ideal $I$ cannot be equal to $I''.$ 
On the other hand $I'=(x^2y, xy^2, x^{\alpha +1}, y^{\beta +1})$ and so $\dim_k(I''/I') \leq 2.$ Since
$I' \subset I \subsetneq I'',$ we see that, again, $\dim_k(I/I')\le 1.$ This concludes the proof of the inequality $\mu(N)\le 2$ 
and the present choice of the elements $x,$ $y.$


Now choose $x, y$ and $L$ as in \autoref{one quadric}(c). Since $\m$ is indecomposable, $I\not=\n^2$ and thus $0:\m\subsetneq \m.$ Every minimal set of generators of $Z$ contains a unit times an element of the form
$L_{0} +\sum_{i>0}\lambda_{i}L_{i},$ whose annihilator is exactly $0:\m$ by \autoref{one quadric}(c). This generator cannot be among the minimal generators of $N',$
so every minimal set of  generators of $N$ contains such an element.

If  $N= A\oplus B$ with $A, B$ not zero,  then $A$ and $B$ must be cyclic because $\mu(N)\le 2.$  We may assume that $A$ is minimally generated by  an element of the form $L_{0} +\sum_{i>0}\lambda_{i}L_{i}$
and thus $A\cong R/(0: \m).$ In particular $\m A \oplus \m B = \m N \cong \m A,$ where the last isomorphism holds by (\ref{importantEQ}). This implies that $\m B = 0$ because the number of generators of $\m B$ is 0.
Since $N$ does not have $k$ as a direct summand, $B=0,$ and we are done.
\end{proof}

If in Theorem \ref{Nindec} the ring $R$ is Gorenstein, that is, a complete intersection, then $a=0.$ Indeed, as explained in the proof above,
$\syz_1^R(\m)$ is indecomposable. Thus $\syz_1^R(\m)=N$ because $N \neq 0.$ Alternatively, one can argue that
$\n(I:\n) = \n I.$ We may assume that $I$ is $\n$-primary, hence generated by a regular sequence $h_1, h_2$ contained in $\n^2.$ As before we write 
 $\begin{pmatrix} h_1 \ \,  h_2 \end{pmatrix}= \begin{pmatrix} x \ \,  y \end{pmatrix} L,$ where $L$ is a $2 \times 2$ matrix with entries in $\n.$ Multiplying this equation
 with the adjoint of $L,$ whose entries are again in $\n,$ one sees that  $\n \Delta  \subset \n I$ with $\Delta = \det(L).$ On the other hand $I :\n=I +(\Delta),$ 
 showing that $\n(I:\n) \subset \n I.$ For more general results along these lines see \cite[Proposition 2.1 and the proof of Theorem 2.2]{CP}.

\begin{proof}[Proof of \autoref{THM1.3}]  By construction the module $N'$ of \autoref{THM1.3} is minimally generated by $a$ elements in the socle of $\syz^{R}_{1} (\m)$ that form part of a minimal generating set of $\syz^{R}_{1} (\m).$ So $N'$ is a direct summand of $\syz^{R}_{1} (\m)$ and $N'\cong k^a,$ with $a$ as in \autoref{Nindec}. Since the number of $k$-summands only depends on $\syz^{R}_{1} (\m),$ the quotient $\syz^{R}_{1} (\m)/N'$ cannot have any $k$-summands and hence is indecomposable
by \autoref{Nindec}. \end{proof}

\section{Proof of Theorem \ref{generalGolod} and applications}\label{proofs}

\begin{proof}[Proof of \autoref{generalGolod}] 
For $j\ge 1$ let $A_j$ be $R$ tensored with the $j^{\rm{th}}$ module in a minimal $S$-free resolution of $R,$ and for $j\ge 2$ set $B_j=A_{j-1}.$ Consider the graded free $R$-module $B=\oplus_{j=2}^{e+1}B_j.$ Write $T=T_{R}(B)$ for the tensor algebra of $B$ over $R$ and $K=K(\m; R)$ for the Koszul complex of $\m.$ As a graded $R$-module, the minimal $R$-free resolution of $k$ is isomorphic to $F: = K\otimes_R T$ because $R$ is Golod.
There are isomorphisms of $R$-modules
$T \cong R\oplus T\otimes_R B,$ and hence $F \cong  K \oplus F\otimes_R B.$ The description of the differential of $F$ in terms of Massey operations shows that $d_F$ is of the form 
\vspace{.1cm}
$$
\begin{tikzcd}[column sep=tiny]
F \arrow{d}{d_F} & \cong & K  \arrow{d}{d_K} & \oplus  & F\otimes B   \arrow{lld} \arrow{d}{d_F\otimes B}\\
F & \cong & K & \oplus   & F\otimes B \, .
\end{tikzcd}
$$

\vspace{.1cm}

\noindent
(see, for instance, \cite[Theorem 5.2.2]{Avramov2010} and its proof). Since $K$ is concentrated in degrees $\le e,$ we obtain,   for $i> e$,
an isomorphism of complexes
\vspace{.1cm}
$$
F_{\ge i}\cong F_{\ge i-e-1} \otimes B_{e+1} \oplus \ \ldots \ \oplus F_{\ge i-2} \otimes B_{2}\, .$$
The assertion now follows because $B_j$ is a free $R$-module of rank $h_{j-1}.$ 
 \end{proof}

\begin{cor}\label{highersyz} Let $(R,\m,k)$ be a Noetherian local ring of embedding dimension $2.$ If $R$ is neither regular nor a zero-dimensional complete intersection, then
\[\syz_3^R(k)\cong k^{\mu(I)-1}\oplus \m^{\mu(I)}\, .
\]\end{cor}
\begin{proof} Since $R$ is Golod, we may apply \autoref{generalGolod}. The result follows because the dimensions of the Koszul homology are
the Betti numbers of $R$ as an $S$-module: $h_{1} = \mu(I)$ and $h_2=\mu(I)-1.$ 
\end{proof}

\begin{thm}\label{main}
Let $(R,\m,k)$ be a Noetherian local ring. If $R$ has embedding dimension $\leq 2$ and is not a zero-dimensional complete intersection, then
every minimal $R$-syzygy of $k$ is a direct sum of copies of $k,$ $\m,$ and $\m^{*}:=\Hom_{R}(\m, R) = \syz_{1}^{R}(\m).$ 
Moreover, copies of at most 3 indecomposable modules are required to build all the syzygies of $k$ as direct sums.
\end{thm}

\begin{proof}
 The first assertion follows from \autoref{generalGolod} and \autoref{syz and dual0}, and the second assertion is a consequence of   \autoref{mindec}, \autoref{mdec}, and \autoref{Nindec}. 
\end{proof}

\begin{thm}\label{decomp of syzygies}
Let $(S,\n, k)$ be a regular local ring of dimension 2, and let $I$ be an ideal contained in $\n^{2}.$ Write $R= S/I$ and $\m = \n R.$ 
\begin{enumerate}[label={$($\alph*\,$)$}]
\item $\syz^{R}_{1}(k)=\m$ is indecomposable if and only if $xy\not\in I$ for any $x,y$ with $\n=(x,y)$ or $R$ is a zero-dimensional complete intersection.
\item $\syz^{R}_{2}(k)$ is indecomposable if and only if $\m$ is indecomposable and $(I:\n)\n = I\n.$  
\item $\syz_3^R(k)$   is indecomposable if and only if $I$ is a principal ideal such that $xy \not\in I$ for any $x,y$ with $\n=(x,y)$ or $R$ is a zero-dimensional complete intersection.
\item $\syz_i^R(k)$ is indecomposable for every $i\ge 0$ if and only if $\syz_3^R(k)$ is indecomposable.
\end{enumerate}
\end{thm}

\begin{proof}

Part (a) follows from \autoref{mindec}. For part (b), notice that if $\m$ is decomposable then $\syz^{R}_{1}(\m)$ is decomposable. Now the assertion follows from \autoref{Nindec}. 

For parts (c) and (d), we may assume that $R$ is neither regular nor a zero-dimensional complete intersection, since otherwise all syzygy modules of $k$ are indecomposable. By \autoref{highersyz} $\syz_3^R(k)$ is indecomposable if and only if $I$ is principal and $\m$ is indecomposable.  Since then $R$ is Gorenstein and $\m$ is a maximal Cohen-Macaulay $R$-module,  $\m$ is indecomposable if and only if one or all of its syzygies are indecomposable. Now part (d) follows and part (c) is a consequence of (a).
\end{proof} 
%
\bibliographystyle{acm}
\bibliography{references}
\end{document}